\documentclass[11pt]{amsart}

\usepackage[active]{srcltx}
\usepackage{amsmath}
\usepackage{amssymb}
\usepackage{amscd}
\usepackage{amsthm}
\usepackage[latin1]{inputenc}
\usepackage{mathrsfs}
\usepackage{pictexwd,dcpic}

\newtheorem{teo}{Theorem}[section]
\newtheorem*{mainteoA}{Theorem A.}
\newtheorem*{mainteoB}{Theorem B.}
\newtheorem*{mainteoC}{Theorem C.}

\newtheorem{prop}[teo]{Proposition}
\newtheorem{cor}[teo]{Corollary}
\newtheorem{lemma}[teo]{Lemma}
\theoremstyle{definition}
\newtheorem{remark}[teo]{Remark}
\newtheorem{say}[teo]{}
\numberwithin{equation}{section}

\newcommand{\ra}{\rightarrow}
\newcommand{\C}{\mathbb{C}}
\newcommand{\R}{\mathbb{R}}
\newcommand{\Zeta}{{\mathbb{Z}}}

\newcommand{\meno}{^{-1}}

\newcommand{\alfa}{\alpha}
\newcommand{\alf}{\alpha}
\newcommand{\vacuo}{\emptyset}

\newcommand{\la}{\lambda}

\newcommand{\restr}[1]          {\vert_{#1}}
\newcommand{\Aut}{\operatorname{Aut}}

\newcommand{\Ann}{\operatorname{Ann}}
\newcommand{\Crit}{\operatorname{Crit}}

\renewcommand{\setminus}{-}
\newcommand{\cinf}{C^\infty}
\newcommand{\om}{\omega}
\newcommand{\eps}{\varepsilon}
\renewcommand{\phi}{\varphi}
\newcommand{\lds}{\ldots}

\newcommand{\cd}{\cdot}
\newcommand{\im}{\operatorname{Im}}
\newcommand{\sx}{\langle}
\newcommand{\xs}{\rangle}

\newcommand{\lra}{\longrightarrow}

\newcommand{\rank}{\operatorname{rank}}

\newcommand{\OO}{\mathcal{O}}

\newcommand{\XX}{\mathcal{X}}

\newcommand{\PP}{\mathbb{P}}

\renewcommand{\phi}             {\varphi}

\newcommand{\debar }            {\bar {\partial } }

\newcommand{\campcoo}      [1]  { \frac {\partial } {\partial #1}}

\newcommand{\teta}{\hat{\eta}}

\newcommand{\seconda}{\operatorname{II}}
\newcommand{\tro}{\tilde{\rho}}

\begin{document}

\author{Elisabetta Colombo, Paola Frediani and Alessandro Ghigi}

\title{On totally geodesic submanifolds in the Jacobian locus}

\address{Universit\`{a} di Milano} \email{elisabetta.colombo@unimi.it}
\address{Universit\`{a} di Pavia} \email{paola.frediani@unipv.it}
\address{Universit\`a di Milano Bicocca}
\email{alessandro.ghigi@unimib.it} 

\thanks{The first author was partially supported by PRIN
    2010-11 MIUR ``Geometria delle Varietà Proiettive".  The second
    and third authors were partially supported by PRIN 2009 MIUR
    ''Moduli, strutture geo\-me\-tri\-che e loro applicazioni''.  The
    second author was partially supported also by FIRB 2012 ''Moduli
    spaces and applications''.  The third author was supported also by
    FIRB 2012 ''Geometria differenziale e teoria geometrica delle
    funzioni'' and also by a grant of Max-Planck Institut f\"ur
    Mathematik, Bonn. The three authors were partially supported by INdAM (GNSAGA)
  } \subjclass[2000]{14H10;14H15;14H40;32G20}

\begin{abstract}
  We study submanifolds of $A_g$ that are totally geodesic for the
  locally symmetric metric and which are contained in the closure of
  the Jacobian locus but not in its boundary.  In the first section we
  recall a formula for the second fundamental form of the period map
  $M_g \hookrightarrow A_g$ due to Pirola, Tortora and the first
  author. We show that this result can be stated quite neatly using a
  line bundle over the product of the curve with itself.  We give an
  upper bound for the dimension of a germ of a totally geodesic
  submanifold passing through $[C] \in M_g$ in terms of the gonality
  of $C$.  This yields an upper bound for the dimension of a germ of a
  totally geodesic submanifold contained in the Jacobian locus, which
  only depends on the genus.  We also study the submanifolds of $A_g$
  obtained from cyclic covers of $\PP^1$. These have been studied by
  various authors.  Moonen determined which of them are Shimura
  varieties using deep results in positive characteristic.  Using our
  methods we show that many of the submanifolds which are not Shimura
  varieties are not even totally geodesic.

\end{abstract}

\maketitle

\tableofcontents{}

\section{Introduction}

\begin{say}
  Denote by $A_g$ the moduli space of principally polarized abelian
  varieties of dimension $g$, by $M_g$ the moduli space of smooth
  curves of genus $g$ and by $j : M_g \ra A_g$ the period mapping or
  Torelli mapping.  {Both} $M_g$ and $A_g$ are complex orbifolds (or smooth
  stacks) {and} $A_g$ is endowed with a locally symmetric metric, the
  so-called Siegel metric.  One expects the Jacobian locus, that is
  the image $j(M_g) \subset A_g$, to be rather curved with respect to
  the Siegel metric. In particular, it should contain very few totally
  geodesic submanifolds of $A_g$.  Another reason for this expectation
  comes from arithmetic geometry. Indeed for a special class of
  totally geodesic submanifolds (Shimura varieties) it has been
  conjectured by Coleman and Oort that for large genus no positive
  dimensional Shimura variety is contained in the closure of the
  Jacobian locus (in $A_g$) and meets the Jacobian locus itself.
  Moonen \cite{moonen-JAG} has proven that an algebraic totally
  geodesic submanifold is a Shimura subvariety if and only if it
  contains a complex multiplication point.  For results on Shimura
  subvarieties contained in $\overline{j(M_g)}$ we refer to
  \cite{hain,moonen-JAG,moonen-oort,rohde,moo,andreatta,mvz,kukulies,moeller}.

  Outside the hyperelliptic locus the period map is an orbifold
  immersion. For $g\geq 4$ the Jacobian locus $j(M_g) $
    has dimension strictly smaller than $A_g$.  Therefore it makes
    sense to compute the second fundamental form of $j(M_g) \subset
    A_g$ and to study its metric properties by infinitesimal methods.
  The second fundamental form has been studied by Pirola, Tortora and
  the first author \cite{cpt}, where an expression for it is given and
  it is proven that the second fundamental form lifts the second
  Gaussian map, as stated in an unpublished paper by Green and
  Griffiths \cite{green}.  In particular the computation of the second
  fundamental form on $\xi_p \odot \xi_p$ (where $\xi_p$ is a Schiffer
  variation at the point $p$ on the curve) reduces to the evaluation
  of the second gaussian map at the point $p$.  These results have
  been used in \cite{cf1} to compute the curvature of the restriction
  to $M_g$ of the Siegel metric.  In \cite{cf1} there is an explicit
  formula for the holomorphic sectional curvature of $M_g$ in the
  direction $\xi_p$ in terms of the holomorphic sectional curvature of
  $ A_g$ and the second Gaussian map.

  It is much harder to use the formula in \cite{cpt} to compute the
  second fundamental form on $\xi_p\odot \xi_q$ when $p\neq q$. In
  fact the formula contains the evaluation at $q$ of a meromorphic
  1-form on the curve, called $\eta_p$, which has a double pole at $p$
  and is defined by Hodge theory. In general it seems rather hard to
  control the behaviour of $\eta_p$, in a way to get constraints on
  the second fundamental form.
\end{say}
\begin{say}
  In this paper we give a global and more intrinsic description of
  this form. We show that as $p$ varies on the curve the forms
  $\eta_p$ glue to give a holomorphic section $\teta$ of the line
  bundle $K_S(2\Delta)$, where $S=C\times C$ and $\Delta \subset S$ is
  the diagonal.  With this interpretation we are able to prove that
  the second fundamental form coincides with the multiplication by
  $\teta$.

  More precisely, fix a genus $g$, which will always be assumed
  greater than 3, and fix $[C]\in M_g$ outside the hyperelliptic
  locus.  The conormal bundle of $j: M_g \subset A_g$ at $[C]$ can be
  identified with $I_2(K_C) $, which is the kernel of the
  multiplication map $S^2H^0(C,K_C) \ra H^0(C,2K_C)$.  Hence the
  second fundamental form can be seen as a map $\rho : I_2(K_C) \ra
  S^2H^0(C,2K_C) $. Let $S=C\times C$ and let $\Delta $ be the
  diagonal.  By K\"unneth formula $H^0(S, K_S) = H^0(C,K_C) \otimes
  H^0(C,K_C) $ and $H^0(S, 2K_S) = H^0(C,2K_C) \otimes H^0(C,2K_C) $.
  In particular $I_2(K_C) \subset H^0(S, K_S(-2\Delta))$.
\end{say}
\begin{mainteoA} [See Theorem \ref{prodotto}]
  The second fundamental form $\rho$ is the restriction to $I_2$ of
  the multiplication map
  \begin{gather*}
    H^0(S, K_S(-2\Delta)) \lra H^0(S,2K_S) \qquad Q \mapsto Q \cd
    \teta .
  \end{gather*}
\end{mainteoA}
\begin{say}Based on these results on the second fundamental form we
  get some constraints on the existence of totally geodesic
  submanifolds of $A_g$ contained in $M_g$. Since our methods are
  local in nature, the results apply to germs of such submanifolds.
  We get upper bounds for the dimension of totally geodesic germs
  passing through $[C]\in M_g$ in terms of the gonality of the curve
  $C$.
\end{say}

\begin{mainteoB}
[See Theorem \ref{stima1}]
Assume that $C$ is a $k$-gonal curve of genus $g$ with $g\geq 4$ and
$k\geq 3$. Let $Y$ be a germ of a totally geodesic submanifold of
$A_g$ which is contained in the jacobian locus and passes through
$j([C])=[J(C)]$. Then $\dim(Y) \leq 2g+k - 4$.
\end{mainteoB}
This immediately yields a bound which only depends on $g$.

\begin{mainteoC}
[See Theorem \ref{stima2}]
If  $g\geq 4$ and $Y$ is a germ of a totally geodesic
submanifold of $A_g$ contained in the jacobian locus, then $\dim Y
\leq \frac{5}{2}(g-1)$.
\end{mainteoC}

\begin{say}
  For low genus one can construct examples of totally geodesic
  submanifolds contained in $M_g$ using cyclic covers of $\PP^1$, see
  e.g. \cite{dejong-noot,moo,rohde}.  These are in fact Shimura
  varieties. A complete list of the Shimura varieties that can be
  obtained in this way has been given in \cite{moo} using deep results
  in positive characteristic.  With our methods we check directly that
  these examples are indeed totally geodesic and we show that a large
  class of cyclic covers, which are not in the list of Shimura
  varieties, are not even totally geodesic (see Proposition
  \ref{suffmoo1} and Corollary \ref{suffmoo2}).
\end{say}
\begin{say}
  Other works studying totally geodesic submanifolds contained in the
  Jacobian locus include \cite{toledo,hain,dejongz}.  In particular
  Hain \cite{hain} proves the following. Let $X$ be an irreducible
  symmetric domain and consider the locally symmetric variety $
  \Gamma\backslash X$ (where $\Gamma $ is a lattice).  If there is a
  totally geodesic immersion $ \Gamma\backslash X \ra
  \overline{j(M_g)} $, and if some additional conditions are
  satisfied, then $X$ must be the complex ball.  De Jong and Zhang
  \cite{dejongz} prove a similar result under milder conditions, but
  still retaining the irreducibility assumption on $X$.  The
  techniques used in these works are global and are based on group
  cohomology and on a rigidity theorem for the mapping class group due
  to Farb and Masur \cite{farb-masur}. Our result instead applies to
  germs of totally geodesic submanifolds, since it is local in nature
  and does not require irreducibility assumptions.  The same local
  point of view is present in \cite{mok}, where the object of study
  are totally geodesic submanifolds contained in an algebraic
  subvariety of a complex hyperbolic space form.
\end{say}

{\bfseries \noindent{Acknowledgements.} }  We wish to thank Fabrizio
Andreatta and Bert van Geemen for interesting conversations.  The
second and third authors wish to thank the Max-Planck Institut f\"ur
Mathematik, Bonn for excellent conditions provided during their visit
at this institution, where part of this paper was written.

\section{Notation and preliminary results}

\begin{say} {\bf Second fundamental form.}
  Denote by $A_g$ the moduli space of principally polarized abelian
  varieties of dimension $g$, by $M_g$ the moduli space of smooth
  curves of genus $g$ and by $j : M_g \ra A_g$ the period mapping or
  Torelli mapping.  By the Torelli theorem $j$ is injective.  To study
  $j$ one can fix a level structure with $n\geq 3$ and consider
  $M_g^{(n)} \stackrel{j^{(n)}}\rightarrow A_g^{(n)}$ which is a
  smooth map between manifolds.  Since level structures play no role
  in what we are doing, it is more appropriate to think of $M_g$ and
  $A_g$ as complex orbifolds or smooth stacks, see e.g. \cite[XII,
  4]{acg}.  The period map is smooth in the orbifold sense.  Moreover
  its restriction to the set of non-hyperelliptic curves is an
  orbifold immersion \cite {os}.  By abuse of terminology we will
  henceforth omit the word "orbifold". {The moduli space}  $A_g$ is endowed with the
  Siegel metric, which is the metric induced on $A_g$ from the
  symmetric metric on the Siegel upper halfspace
  $\mathfrak{H}_g$. Outside the hyperelliptic locus we have the
  sequence of tangent bundles:
  \begin{equation}
    \nonumber
    0\rightarrow { T}{M_g} \rightarrow
    j^{*}\bigl ({T}{A_g}\bigr ) \stackrel{\pi}\rightarrow
    { N} \rightarrow 0,
  \end{equation}
  whose dual, at $[C] \in M_g$ is
  \begin{equation}
    \nonumber
    0 \rightarrow { I}_2 \rightarrow S^2
    H^0(C,K_C) \stackrel{m}\rightarrow H^0(C,2K_C)
    \rightarrow 0,
  \end{equation}
  where $I_2:={ I}_2 (K_C) $ is the set of quadrics containing the
  canonical curve and $m$ is the multiplication map (see \cite{cf1}
  for more details).  Denote by
  \begin{gather*}
    \seconda : S^2 T_{[C]}M_g=S^2H^1(C,T_C) \ra {N}_{[C]}
  \end{gather*}
  the second fundamental form of the period map with respect to the
  Siegel metric on $A_g$.  Denote by
  \begin{gather*}
    \rho : I_2\ra S^2 H^0(C, 2K_C)
  \end{gather*}
  the dual of $\seconda$.  We will refer both to $\seconda$ and to
  $\rho$ as second fundamental forms.
\end{say}

\begin{say}{\bf Schiffer variations.}
  If $C$ is a curve and $x\in C$, the coboundary of the exact sequence
  $0 \ra T_C \ra T_C(x) \ra T_C(x)\restr{x} \ra 0$ yields an injection
  $H^0(T_C(x)\restr{x}) \cong \C \hookrightarrow H^1(C,T_C)$. Elements
  in the image are called \emph{Schiffer variations} at $x$.  If
  $(U,z)$ is a chart centred at $x$ and $b\in \cinf_0(U)$ is a bump
  function which is equal to 1 on a neighbourhood of $x$, then
  \begin{gather*}
    \theta:= \frac{\debar b}{z}\cd \campcoo{z}
  \end{gather*}
  is a Dolbeault representative of a Schiffer variation at $x$.  The
  map
  \begin{gather*}
    \xi : TC \ra H^1(C,T_C) \qquad u=\la\campcoo{z} (x)\mapsto \xi_u:=
    \la^2 [\theta]
  \end{gather*}
  does not depend on the choice of the coordinates.  It is well known
  that Schiffer variations generate $H^1(C,T_C)$ \cite[p.175]{acg}.
\end{say}
\begin{lemma}
  \label{cauchy}
  Let $\beta \in H^0(C, 2K_C)$ and let $(U,z)$ be a chart centred at
  $x\in C$. If $\beta = f(z) (dz)^2$ on $U$, then $ \beta \bigl (
  \xi_{\campcoo{z}(x)} \bigr ) = 2\pi i\, f(0).$
\end{lemma}
\begin{proof}
  \begin{gather*}
    \beta \bigl ( \xi_{\campcoo{z}(x)} \bigr ) = \int_C \beta \cup
    \theta = \int_U f(z) dz \wedge \frac{\debar b}{z} =
    -\int_{U\setminus\{x\}} \debar \biggl ( \frac{b (z)f(z) }{z} dz
    \biggr ) .
  \end{gather*}
  If $\eps>0$ is small enough, $b\equiv 1 $ on $\{|z|\leq \eps\}$.
  Using Stokes and Cauchy theorems we get
  \begin{gather*}
    \beta \bigl ( \xi_{\campcoo{z}(x)} \bigr ) = -\lim_{\eps \to 0}
    \int\limits_{U\cap \{|z| > \eps \}} \debar \biggl ( \frac{b
      (z)f(z) }{z} dz \biggr ) = \lim_{\eps \to 0}
    \int\limits_{|z|=\eps} \frac{f(z) }{z} dz = 2\pi i f(0).
  \end{gather*}
\end{proof}

\begin{say}\label{saymu}
{\bf Gaussian maps.}
We briefly recall the definition of Gaussian maps for curves. Let $N$
and $M$ be line bundles on $C$.  Set $S:=C\times C$ and let $\Delta
\subset S$ be the diagonal. For a non-negative integer $k$ the
\emph{k-th Gaussian} or \emph{Wahl map} associated to these data is
the map given by restriction to the diagonal
\begin{equation*}
  H^0(S,    N\boxtimes M (-k \Delta) ) \stackrel{    \mu^k_{N,M}}
  {\longrightarrow}
  H^0(S, N\boxtimes M (-k \Delta) \restr{\Delta})\cong
  H^0(C, N\otimes M \otimes K_C^{{k}}).
\end{equation*}
We are only interested in the case $N=M$.  In this case we set
$\mu_{k,M}:=\mu^k_{M,M}$.  With the indentification $H^0(S, N
\boxtimes M) \cong H^0(C,N) \otimes H^0(C,M)$ the map $\mu_{0,M}$ is
the multiplication map of global sections
\begin{equation*}
  H^0(C,M)\otimes
  H^0(C,M)\rightarrow H^0(C,M^{ 2}),
\end{equation*}
which obviously vanishes identically on $\wedge^2 H^0(C,M)$.
Consequently $\ker \mu_{0,M}$ $= H^0(S, M\boxtimes M(-\Delta))$
decomposes as $\wedge^2 H^0(C,M)\oplus I_2(M)$, where $I_2(M)$ is the
kernel of $S^2H^0(C,M)\rightarrow H^0(C,M^2)$. Since $\mu_{1,M}$
vanishes on symmetric tensors, one usually writes
\begin{equation}\nonumber
  \mu_{1,M}:\wedge^2H^0(M)\rightarrow H^0(
  K_C\otimes M^{ 2}).
\end{equation}
If $\sigma$ is a local frame for $M$ and $z$ is a local coordinate,
given sections $s_1, s_2 \in H^0(C,M)$ with $s_i = f_i(z) \sigma$, we
have
\begin{gather}
  \label{muformula}
  \mu_{1,M} (s_1\wedge s_2) = (f'_1 f_2 - f'_2f_1) dz \otimes \sigma^{
    2} .
\end{gather}
Consequently, the zero divisor of $\mu_{1,M} (s_1\wedge s_2)$ is twice
the base locus of the pencil $\sx s_1, s_2 \xs $ plus the ramification
divisor of the associated morphism.

Again $H^0(S, M\boxtimes M (-2\Delta))$ decomposes as the sum of
$I_2(M)$ and the kernel of $\mu_{1,M}$. Since $\mu_{2,M}$ vanishes
identically on skew-symmetric tensors, one usually writes
\begin{equation*}
  \mu_{2,M}: I_2(M)\rightarrow H^0(C,M^2\otimes K_C^2).
\end{equation*}
By $\mu_2$ we denote the second gaussian map of the canonical line
bundle $K_C$ on $C$:
\begin{equation*}
  \mu_2:= \mu_{2,K_C}:I_2(K_C)\rightarrow H^0(K_C^4).
\end{equation*}
\end{say}

\begin{say} {\bf The form $\eta_x$ and the second fundamental form.}
  We now recall the definition of $\eta_x$.  Let $C$ be a smooth
  complex projective curve of genus $g\geq 4$. Fix a
  point $x \in C$.  The space $H^0(C, K_C (2x))$ is contained in the
  space of closed 1-forms on $C\setminus \{x\}$. The induced map
  $H^0(C,K_C(2x)) \ra H^1(C \setminus \{x\},\C)$ is injective as soon
  as $g>0$. By the Mayer-Vietoris sequence, the inclusion $C\setminus
  \{x\} \hookrightarrow C$ induces an isomorphism $H^1(C,\C) \cong
  H^1(C\setminus \{x\}, \C)$. Thus we get an injection
  \begin{gather}
    \label{jx}
    j_x : H^0(C,K_C(2x)) \hookrightarrow H^1(C,\C).
  \end{gather}
  $H^{1,0}(C)$ is contained in the image of $j_x$ and $h^0(C,K_C(2x))
  = g+1$, so $ j_x\meno ( H^{0,1}(C))$ is a line.  If $(U,z)$ is a
  chart centred at $x$, there is a unique element $\phi$ in this line
  such that on $U\setminus \{x\}$
  \begin{gather*}
    \phi = \biggl (\frac{1}{z^{2}} + h(z) \biggr ) dz
  \end{gather*}
  with $h\in \OO_C(U)$. (Applying Stokes theorem on $C$ minus a disc
  around $x$ shows that the residue at $x$ vanishes, so there is no
  term in $1/z$.)  Define a linear map
  \begin{equation}
    \nonumber
    \eta_x : T_xC \ra
    H^0(C,K_C(2x))
  \end{equation} by the rule
  \begin{gather*}
    u=\la \campcoo{z}(x) \longmapsto \eta_x(u) : = \la\phi.
  \end{gather*}
  We will often drop $x$ and simply write $\eta_u$ for $\eta_x(u)$.
  An easy computation shows that $\eta_x $ does not depend on the
  choice of the local coordinate.
\end{say}

\begin{teo}
  [\protect{\cite[Thm. 3.1]{cpt}, \cite[Lemma 3.5]{cf1}}]
  \label{copito}
  Let $C$ be a non-hyperelliptic curve of genus $g\geq 4$. Given
  points $x\neq y$ in $ C$ and tangent vectors $u\in T_xC$ and $ v\in
  T_yC$ we have
  \begin{gather*}
    \rho(Q)(\xi_u \odot\xi_v ) 
    = -4 \pi i \cd \eta_x(u)( v) Q(u,v) ,\\
    \rho(Q)(\xi_u \otimes\xi_u ) 
    = -2\pi i \cd \mu_2 (Q)(u^{\otimes 4}).
  \end{gather*}
\end{teo}
\noindent This theorem is basic to the whole paper.  For the reader's
convenience we recall its proof.

\begin{proof} 
  First of all we take $w\in H^1(C,T_C)$ and we compute $\rho(Q)(w)\in
  H^0(C,2K_C)$ for $Q\in I_2(K_C )$.  We can assume that there is a a
  one dimensional deformation ${\mathcal X} \stackrel{f} \rightarrow
  \Delta$ with $C= f^{-1}(0)$ and $w=\kappa(\frac{\partial}{\partial
    t})$, where $\Delta= \{|t|<1\}$ and $\kappa$ is the
  Kodaira-Spencer map. Take a ${\mathcal C}^{\infty}$ lifting $Y$ of
  the holomorphic vector field $\frac{\partial}{\partial t}$.  So we
  have a ${\mathcal C}^{\infty}$ trivialization $\tau: \Delta \times C
  \rightarrow {\mathcal X}$, $\tau(t,x) =\tau_t(x) := \Phi_{t}(x)$,
  where $\{\Phi_t\}$ is the flow of the vector field $Y$. Then
  $\theta:= \overline{\partial} Y\restr{C}$ is a $\debar$-closed form
  in $A^{0,1}(C,T_C)$ such that $ [\theta]=w\in H^1(C,T_C)$. Denote by
  $C_t$ the fibre of $f$ over $t$ and let $\omega(t)$ be a section of
  the Hodge bundle, i.e.  $\forall t \in \Delta$, $\omega(t) \in
  H^0(K_{C_t})$.  Since $\omega(t)$ is closed, also
  $\tau_t^*(\omega(t))$ is closed, so $\tau_t^{*}(\omega(t))= \omega +
  (\alpha + dh)t +o(t),$ where $\omega := \omega(0)$, $\alpha$ is
  harmonic and $h$ is a ${\mathcal C}^{\infty}$ function.  Denote by
  $\nabla^{GM}$ the flat Gauss-Manin connection on $R^1f_*\C$.  So we
  have $ \nabla^{GM}_{\partial/\partial t}[\omega(t)]\restr{t=0} =
  [\alpha].  $ By Griffiths' results (see
  e.g. \cite[pp. 234ff]{voisin-theorie}) $\theta \cdot \omega =
  \alpha^{0,1} + \overline{\partial}h $ and
  $\kappa(\frac{\partial}{\partial t}) \cdot [\omega] = [
  \alpha^{0,1}]$, where $ \alpha^{0,1}$ is the $(0,1)$ component of
  $\alpha$.  Now assume that $\{\omega_i\}_{i=1,...,g}$ is a basis of
  $H^0(K_C) $. Take a quadric $Q = \sum_{i,j} a_{ij} \omega_i \otimes
  \omega_j \in I_2(K_C)$.  Denote by $\nabla^{1,0}$ the Gauss-Manin
  connection on the Hodge bundle $f_* \om_{\XX/\Delta}$,
  i.e. $\nabla^{1,0} = \pi \nabla^{GM}$, where $\pi$ is the projection
  of $H^1(C_t,\C)$ onto $H^0(C_t,K_{C_t})$.  Then for all $ i$ we have
  $\nabla^{1,0}_{\partial/\partial t}[\omega_i(t)]\restr{t=0} =
  [{\alpha_i}^{1,0}]$.  Denote by $\nabla$ the induced connection on
  $S^2f_*\om_{\XX/\Delta}$.  If $\tilde{Q}(t) = \sum_{i,j} a_{ij}(t)
  \omega_i(t) \otimes \omega_j(t) \in I_2(K_{C_t})$ is a section of
  the conormal bundle such that $\tilde{Q}(0) = Q$, then
  \begin{gather*}
    \rho(Q)(w) = m ( \nabla_{\frac{\partial}{\partial t}}
    \tilde{Q}\restr{t=0}) = \sum_{i,j} a'_{ij}(0) \omega_i \omega_j +
    2 \sum_{i,j} a_{ij} \alpha^{1,0}_i \omega_j .
  \end{gather*}
  Since $\sum_{i,j} a_{ij}(t) \omega_i(t) \omega_j(t) \equiv 0$, also
  its derivative at $t=0$ vanishes, i.e. $2\sum_{i,j} a_{ij}(\alpha_i
  + d h_i) \omega_j + \sum_{i,j} a'_{ij}(0) \omega_i \omega_j \equiv
  0$, and if we take the $(1,0)$ part we have $2\sum_{i,j}
  a_{ij}(\alpha^{1,0}_i + \partial h_i) \omega_j + \sum_{i,j} a'_{ij}
  \omega_i \omega_j \equiv 0,$ so
  \begin{equation}
    \label{Hodge-Gauss}
    \rho(Q)(w) = -2 \sum_{i,j} a_{ij}
    \omega_j \partial h_i.   
  \end{equation}
  This is an instance of a Hodge-Gaussian map, see \cite{cpt} and also
  \cite[\S 4]{pi}.

  Now fix a point $x \in C$ and a chart $(U,z)$ centred at $x$.  Let
  $w$ be the Schiffer variation $\xi_{\campcoo{z}(x)}$ at $x \in C$
  with Dolbeault representative $\theta:= \frac{\debar b}{z}\cd
  \campcoo{z}$, where $b\in \cinf_0(U)$ be a bump function equal to 1
  on a neighbourhood of $x$. Let $\omega_i =f_i(z) dz$ be the local
  expression of $\omega_i$ in $U$.  On $C\setminus \{x\}$ we have
  $\theta \cdot \omega_{i}=\frac{f_{i}}{z}\bar{\partial}b$, so
  $\alpha_{i}^{0,1}+\bar{\partial}h_{i}=
  \bar{\partial}\left(\frac{bf_{i}}{z}\right)$.  Set
  $g_{i}:=\frac{bf_{i}}{z}-h_{i}$.  Then
  $\alpha_{i}^{0,1}=\bar{\partial}g_{i}$.  Define $\eta_{i}:=\partial
  g_{_i}$.

  Now we will show that $ \sum a_{{ij}}\omega_{i}\partial h_{j}=- \sum
  a_{{ij}}\omega_{_i}\eta_{j}.$ In the first place, note that
  $\eta_{i}$ is holomorphic in $C\setminus \{x\}.$ Indeed,
  $\alpha_{i}$ is harmonic, thus
  $\bar{\partial}\eta_{i}=\bar{\partial}\partial g_{i}=
  -\partial\bar{\partial}g_{i}=-\partial\alpha_{i}^{0,1}=0$. Hence,
  $\sum a_{_{ij}}\omega_{i}\partial(\frac{bf_{j}}{z})=\sum
  a_{{ij}}\omega_{i}\partial h_{j}+ \sum a_{{ij}}\omega_{i}\eta_{j}$
  is holomorphic on $C\setminus \{x\}$, because the first term is a
  holomorphic section of $2K_C$ by \eqref{Hodge-Gauss} and the second
  is holomorphic on $C\setminus \{x\}$ since $\eta_j $ is holomorphic.
  In a neighborhood of $x$ where $b\equiv1,$ this expression has the
  form
  \begin{gather*}
    \sum a_{{ij}}f_{i}\frac{\partial}{\partial z}
    \left(\frac{f_{j}}{z}\right)dz^2= \sum
    a_{_{ij}}f_{i}\left(-\frac{f_{j}}{z^2}+
      \frac{{f'_{j}}}{z}\right)dz^2=0,
  \end{gather*}
  because $\sum a_{{ij}}f_{i}f_{j}=\sum a_{{ij}}f_{i}{f'_{j}}=0$.  So
  $\sum a_{_{ij}}\omega_{i}\partial(\frac{bf_{j}}{z})$ is identically
  zero.  Thus we have
  \begin{gather*}
    \rho(Q)(\xi_{\campcoo{z}(x)})= -2 \sum a_{{ij}}\omega_{i}\partial
    h_{j}=2 \sum a_{{ij}}\omega_{i}\eta_{j}.
  \end{gather*}

  Now we claim that $\eta_i = - f_i(x) \eta_x(\frac{\partial}{\partial
    z})$. In fact $\eta_{_i}\in H^0(K_C(2x)),$ since on $U,$ where
  $b\equiv1$, $\eta_{i}$ has the form
  \begin{equation}
    \nonumber
    \eta_{i}=\left(-\frac{f_{i}(x)}{z^2}+\psi_{i}(z)\right)dz
  \end{equation}
  with $\psi_{i}(z)$ a holomorphic function, hence $\eta_{i}$ is a
  meromorphic form, with a double pole at $x$. By definition,
  $\eta_{i}+\alpha_{i}^{0,1}=\partial
  g_{i}+\bar{\partial}g_{i}=dg_{i},$ so
$$j_x(\eta_{i})=-[\alpha_{i}^{0,1}]\in  j_x(H^0(K_C(2x)))\cap H^{0,1}(C). $$
This proves the claim.

We can assume that the chart $(U,z)$ contains also $y$. Applying Lemma
\ref{cauchy} the first statement immediately follows for
$u=\partial/\partial z (x)$ and $v= \partial/\partial z(y)$.  It is
clear that this is enough. For the second statement it suffices to use
the local expression of $\mu_2$.
\end{proof}

\begin{remark} \label{rem-notazione} We remark that Schiffer
  variations, the forms $\eta_x$, $Q$ and $\mu_2(Q)$ are sections of
  vector bundles, but they become functions as soon as a coordinate
  chart is fixed. Because of this many statements, like the one above,
  are usually stated for simplicity, as if these sections were
  evaluated at points instead of tangent vectors. We will follow this
  notation when it is convenient. In the next section instead it is
  better to stick a more formal notation.
\end{remark}

\section{The second fundamental form as a multiplication map}

\begin{say}
  In this section we show that as $x$ varies on the curve the form
  $\eta_x$ varies holomorphically in an appriopriate sense and gives
  rise to a section of a vector bundle on $C$ and to a corresponding
  section $\teta$ of the line bundle $K_S (2\Delta)$ on $S=C\times
  C$. The two main points are Theorem \ref{prodotto} and the
  invariance of $\teta$ with respect to the action of $\Aut(C)$.
\end{say}
\begin{say}
\label{def-elementary}
  Let $(U,z)$ be a chart centred at $x\in C$. Set
  $u=\campcoo{z}(x)$. It is a classical result that there is a
  harmonic function $f_u \in \cinf(C\setminus \{x\} )$ such that $ f_u = - 1/z
  + g(z)$ on $U \setminus  \{x\}$ for some $g\in \cinf (U)$. This function is
  unique up to an additive constant and is called \emph{elementary
    potential}.  Its existence can be proven for example using the
  (real) Hodge decomposition theorem and the Weyl lemma (see e.g.
  \cite[p. 46-48]{farkas-kra}) or using the Perron method (see
  e.g. \cite[p. 213ff]{lamotke-Riemannsche}).
\end{say}

\begin{lemma} \label{jeta} If $f_u$ is an elementary potential, then
  $\partial f_u = \eta_u $, $\debar f_u $ is smooth on $C$ and
  $j_x(\eta_u) = [-\debar f_u]$.
\end{lemma}
\begin{proof} The $(1,0)$-form $\partial f_u$ is holomorphic on $C -
  \{x\}$ since $f_u$ is harmonic. Moreover $
  \partial f_u = z^{-2} dz + \partial g$ on $U \setminus  \{x\}$. The form
  $\partial g$ is holomorphic on $U\setminus \{x\}$, but also smooth on
  $U$. Hence it is holomorphic on $U$. This shows that $\partial f_u =
  \eta_u + \om$ for some $\om \in H^0(K_C)$. Set $\alfa : = - \debar
  f_u$.  Then $\alfa$ is smooth on $C$ by the definition of $f_u$.  It
  is closed, since $f_u$ is harmonic and of type $(0,1)$. On $C
  \setminus \{x\}$ we have $ \partial f_u - \alf = df_u$, so $[\partial f_u ] =
  [\alfa] $ in $H^1(C\setminus \{x\})$. Therefore $j_x(\partial f_u ) =
  [\alf]$. Since $[\alf] \in H^{0,1}(C)$, this shows that
  $j_x(\partial f_u) \in H^{0,1}(C)$. Therefore $\om=0$ and $\eta_u
  = \partial f_u$.
\end{proof}

\begin{remark}
  Notice that one could prove the existence of $f_u$ using $\eta_u$
  and the fact that $C\setminus \{x\}$ is Stein.
\end{remark}

\begin{lemma}\label{eta-can}
  If $H^{0,1}(C)$ is identified with $H^0(C,K_C)^*$ using Serre
  duality, then $x \mapsto \im j_x\cap H^{0,1}(C) \in \PP(
  H^{0,1}(C))$ coincides with the canonical map.
\end{lemma}
\begin{proof}
  Let $(U,z)$ be a coordinate centred at $x$. Set
 $u=\campcoo{z} (x)$.
  If $\om \in H^0(C,K_C)$, let $\om=\phi(z)dz$ be its local expression
  on $U$.  By Lemma \ref{jeta} $j_x(\eta_u) = -[\debar f_u]$.  Therefore
  \begin{gather}
    \label{eq:eta-can}
    \begin{gathered}
      \int_C \om \wedge j_x(\eta_u) = -\int_{C} \om \wedge \debar f_u
      = \int_{C\setminus \{x\}} d ( f_u\, \om) = \\
      = \lim_{\eps \to 0} \int_{|z|=\eps} \phi(z) \Bigl ( -\frac{1}{z}
      + g(z) \Bigr ) dz = -\phi(0) = - \om (u).
    \end{gathered}
  \end{gather}
\end{proof}

\begin{say}
  Let $S:= C\times C$, let $\Delta$ denote the diagonal and let $p, q:
  S \ra C$ be the projections $p(x,y) =x$, $q(x,y) = y$.  Then $ K_S =
  p^*K_C \otimes q^*K_C$. Consider the line bundle $L:= K_S(2\Delta)$
  on $S$ and set
  \begin{gather*}
    V:=p_*(q^*K_C ( 2\Delta)) \qquad E:=p_* L.
  \end{gather*}
  By the projection formula $E = K_C \otimes V$. Since $ q^*K_C (2
  \Delta)\restr{\{x\}\times C } = q^* K_C (2x) $, we have $H^0(
  p\meno(x) , q^*K_C (2\Delta)) \cong H^0 (C,K_C (2x)) $.  By Grauert
  semicontinuity theorem $V$ is a holomorphic vector bundle on $C$
  with fibre $V_x \cong H^0( C , K_C ( 2x)) $ and the map $x\mapsto
  \eta_x$ is a section of $ E$.  We call this section
  $\eta$. \end{say}
\begin{prop}
  $\eta$ is a holomorphic section of $E$.
\end{prop}
\begin{proof}
  Let $W \ra C$ denote the trivial vector bundle with fibre
  $H^1(C,\C)$.  We claim that the injection $j: V \hookrightarrow W$
  defined in \eqref{jx} is holomorphic.  Fix $\alfa \in
  H^0(C,K_C(2x))$ and a smooth singular 1-cycle $c $ on $C$. If $x$
  does not lie in the support of $c$, the integral $ \int_c \alfa $ is
  well-defined.  It does not change if $\alfa$ is replaced by $\alfa +
  df$ with $f\in \cinf (C \setminus \{x\})$. Therefore $ \int_c \alfa
  = \sx [c], j(\alfa)\xs $.  Fix $x_0 \in C$ and choose smooth
  1-cycles $c_1 ,\lds, c_{2g}$, that do not touch $x_0$ and whose
  classes form a basis of $H_1(C, \C)$.  Let $A$ be an open subset of
  $ C$, such that $V$ is trivial over $A$ and $\overline{A}$ does not
  intersect the supports of the cycles $c_i$.  Fix $s\in H^0(A, V)$. To show
  that $j(s)$ is a holomorphic section of $W$ over $A$, it is enough
  to prove that the functions
  \begin{gather*}
    x \longmapsto \sx [c_i], j_x(s(x))\xs = \int_{c_i} s(x)
  \end{gather*}
  are holomorphic on $A$.  Since $V=p_*q^*K_C (2\Delta)$, $s$
  corresponds to a section $s \in H^0(A\times C, q^*K_C (2\Delta))$.
  So $s(x, \cd) \restr{C\setminus A}$ is a holomorphic 1-form on
  $C\setminus A$ depending holomorphically on the parameter $x\in A$
  and its integral over the 1-cycle $c_i$ (which is contained in
  $C\setminus A$) is a holomorphic function of $x$.  Therefore $j(s)$
  is holomorphic. Since $s$ is arbitrary, this proves that $j$ is
  holomorphic, as claimed.  Next fix a chart $(U,z)$ on
  $C$. To show that $\eta $ is holomorphic on $U$ it is enough to
  prove that $\eta({\campcoo{z}}) $ is holomorphic or - by the above -
  that $j(\eta({\campcoo{z}}))$ is a holomorphic function $U\ra
  H^1(C,\C)$.  Fix a basis $\{\om_1, \lds, \om_g\}$ of $H^0(C,K_C)$.
  The functionals $ \int_C (\cd )\wedge \om_j$ and $ \int_C (\cd
  )\wedge \overline{\om}_j$ form a basis of $H^1(C,\C)^*$. Since
  $j(\eta(\campcoo{z})) \in H^{0,1}(C)$ the latter $g$ functionals
  vanish on $j(\eta(\campcoo{z}))$.  Assume that $\om_j(z) = f_j(z)
  dz$ on $U$. By \eqref{eq:eta-can}
  \begin{gather*}
    \int_C j\Big(\eta(\campcoo{z})\Big) \wedge \om_j= f_j(z).
  \end{gather*}
  This proves that $j(\eta(\campcoo{z}))$ is holomorphic.
\end{proof}

\begin{say}
  \label{teta-eta}
  Since $E=p_* L$ there is an isomorphism $ H^0(C,E) \cong H^0(S,L)$
  that associates to $\alfa \in H^0(C,E)$ the section $\hat{\alf} \in
  H^0(S,L)$ such that
  \begin{gather*}
    \alf_x= \hat{\alf} \restr{\{x\}\times C} \in T_x^*C \otimes
    H^0(C,K_C(2x)) = E_x.
  \end{gather*}
  It follows that
  $    \alf_x (u)(v)=$ $\hat{\alfa}\bigl ( (u,0),(0,v)\bigr)$
for $x \neq y$,  $u\in T_xC$ and $v\in T_y C$.

Thus
  there is a well-defined section $\teta \in H^0(S,L)$ corresponding
  to $\eta$ and for $u\in T_xC$ and $v\in T_y C$ with $x\neq y$, we
  have $ \eta_x (u)(v)=\teta(u,v) $.
\end{say}

\begin{lemma}
  \label{sym}
  The form $\teta$ is symmetric, i.e.  $\teta \bigl ((u, 0)
  ,(0,v)\bigr ) = \teta \bigl ((v, 0) ,(0,u)\bigr ) $.
\end{lemma}
\proof
Fix points $x \neq y$ and tangent vectors $u\in T_xC, y\in T_yC$.
Using the identity $ d \Bigl ( f_u ( \debar f_ v - \partial f_v) - f_v
( \debar f_u - \partial f_u) \Bigr ) = 2 \bigl ( f_u \partial \debar
f_v - f_v \partial \debar f_u \bigr ) = 0 $ and applying Stokes
theorem on $C \setminus  \{|z| < \eps\} \cup \{|w|< \eps\}$ we get
\begin{equation}
  \nonumber
  \begin{gathered}
    0 = \int_{|z|=\eps} \Bigl ( f_u ( \debar f_ v - \partial f_v) -
    f_v ( \debar f_u - \partial f_u) \Bigr )
    +\\
    + \int_{|w|=\eps} \Bigl ( f_u ( \debar f_ v - \partial f_v) - f_v
    ( \debar f_u - \partial f_u) \Bigr )
  \end{gathered}
\end{equation}
(This is just Green formula.)  Let us denote by $A_\eps$ and $B_\eps$
these two integrals.  Observe that $ f_v \partial f_u = \partial (f_u
f_v) - f_u \partial f_v$ and $ f_u \debar f_v = \debar (f_u f_v) - f_v
\debar f_u $. Therefore
\begin{gather*}
  A_\eps=\int_{|z| =\eps} \Bigl ( d(f_u f_v) - 2 f_v \debar f_u - 2
  f_u \partial f_v \Bigr) =-2 \int_{|z|=\eps} \Bigl ( f_v \debar f_u
  +f_u \partial f_v \Bigr) .
\end{gather*}
Since $f_v \debar f_u$ is a smooth form near $x$, we have $ \lim_{\eps
  \to 0} \int_{|z|=\eps} f_v \debar f_u = 0 $. Choose a coordinate
$(U,z)$ centred at $x$ such that $u=\partial /\partial z (x)$ and let
$g(z)$ be as in \ref{def-elementary}.  Near $x$ we have $ \eta_y(v)
= \partial f_ v = h(z) dz$ for some holomorphic function $h$. So
\begin{gather*}
  \lim_{\eps\to 0} A_\eps = -2 \lim_{\eps \to 0} \int_{|z|=\eps}
  f_u \partial f_ v = \\
  =- 2 \lim_{\eps \to 0} \int_{|z|=\eps} \biggl ( -\frac{1}{z} + g(z)
  \biggr ) h(z) dz = 4 \pi i h(0) = 4\pi i\eta_y(v)(u).
\end{gather*}
By the corresponding computation, $\lim_{\eps \to 0} B_\eps $ $=$ $-4 \pi i
\eta_x(u) (v)$, so $ \eta_y(v) (u) = \eta_x(u)(v)$ as desired.  \qed

\begin{say}
  $\Aut(C)$ acts diagonally on $C\times C$ and preserves
  $\Delta$. Therefore the action lifts to $K_S$ and also to
  $L=K_S(2\Delta)$. This yields a representation of $\Aut(C)$ on
  $H^0(S,L)$.  On the other hand, if $\sigma\in \Aut(C)$, then
  $(\sigma\meno)^*$ is a map from $H^0(C, K_C(2x) ) = V_x $ to
  $H^0(C,K_C(2\sigma(x)) )= V_{\sigma(x)}$. Tensoring it with
  $(d\sigma\meno)^* : T_x^*C \ra T_{\sigma(x)}^*C$ we get a map
  $T^*_xC \otimes V_x \ra T_{\sigma(x)}^*C \otimes V_{\sigma(x)}$.
  This yields an action of $\Aut(C)$ on the total space of
  $E=K_C\otimes V$ which covers the action of $\Aut(C)$ on $C$. This
  means that $E$ is an equivariant bundle. In this way we get a
  representation of $\Aut(C)$ on $H^0(C,E)$. The map $\alf \mapsto
  \hat{\alf}$ considered in \ref{teta-eta} is an isomorphism of
  $\Aut(C)$--representations.
\end{say}

\begin{lemma}
  \label{invarho}
  $\eta$ and $\teta $ are invariant with respect to the action of
  $\Aut(C)$.
\end{lemma}
\begin{proof}
  By the above it is enough to show that $\eta$ is invariant.  For
  $\tau \in \Aut(C)$ we wish to prove that $(d\tau\meno)^*\otimes
  (\tau\meno)^* (\eta_y)= \eta_{\tau(y)}$ for any $y\in C$.  For
  simplicity, set $\sigma : = \tau\meno$, $x:=\tau(y)$. So we wish to
  prove that $\sigma^* (\eta_y({d\sigma(u))}) = \eta_x(u)$
  for any $u\in T_xC$. By continuity it is enough to prove this for $x
  $ such that $ \sigma(x) \neq x$. Choose a coordinate patch $(U,z)$
  centred at $x$, such that $U\cap \sigma(U) \neq \vacuo$ and
  $u=\campcoo{z}\vert_x$. Then $(\sigma(U), w:=z\circ \sigma\meno)$ is
  a coordinate system centred at $\sigma(x)$ and
  $\campcoo{w}(\sigma(x)) = d\sigma(u)$. Assume that
  \begin{gather*}
    \eta_x(u) = \biggl ( \frac{1}{z^2} + h(z) \biggr ) dz
  \end{gather*}
  on $U$. Then
  \begin{gather*}
    \tau^*\eta_x(u) = \biggl( \frac{1}{w^2} + h(w) \biggr )
    dw
  \end{gather*}
  on $\sigma(U)$.  Moreover $ j_y \tau^* = \tau^*j_x$. So $ j_y
  \tau^*( \eta_x(u)) =\tau^*j_x(\eta_x(u ))
  $. Since $j_x(\eta_x(u) ) \in H^{0,1}(C)$, also $ j_y \tau^*(
  \eta_x(u)) \in H^{0,1}(C)$. Hence $ \tau^*\eta_x(u) = \eta
  _y(d\sigma(u))$ as desired.
\end{proof}

\begin{say}
  If we identify $H^0(C, K_C) \otimes H^0(C,K_C)$ with $H^0(S,K_S) $,
  then $I_2 $ becomes a subspace of $H^0(S, K_S( - \Delta))$.  Since
  elements of $I_2$ are symmetric, they are in fact in
  $H^0(S,K_S(-2\Delta))$. So if $Q\in I_2$ the product section $Q\cd
  \teta$ lies in $H^0(S,2K_S) \cong H^0(C, 2K_C) \otimes H^0(C,
  2K_C)$.
\end{say}
\begin{teo}
  \label{prodotto} With the above identifications, if $C$ is
  non-hyperelliptic and of genus $g\geq 4$, then $\rho : I_2 \ra
  S^2H^0(C,2K_C)$ is the restriction to $I_2$ of the multiplication
  map
  \begin{gather*}
    H^0(S, K_S(-2\Delta)) \lra H^0(S,2K_S) \qquad Q \mapsto Q \cd
    \teta .
  \end{gather*}
\end{teo}
\begin{proof}
  The identification $ H^0(C, 2K_C) \otimes H^0(C, 2K_C) \cong H^0(S,
  2K_S) $ is compatible with Lemma \ref{cauchy}, i.e.  for $\alf \in
  H^0(S,2K_S)$, $u\in T_xC$ and $v\in T_yC$ we have $ \sx \alf, \xi_u
  \otimes \xi_v\xs = \alfa\bigl((u,0),(0, v)\bigr)$.  If $x\neq y$, by
  Theorem \ref{copito}
  \begin{gather*}
    \rho(Q) ( \xi_u\odot \xi_v) = Q(u, v) \eta_x(u)(v) = \\
     = (Q\cd \teta) 
\bigl ((u,0), (0,v)\bigr )    =
(Q\cd \teta) \, ( \xi_u \odot \xi_v ).
  \end{gather*}
  (In the last identity we use the fact that both $Q$ and $\teta$ are
  symmetric).  So $\rho(Q) - Q\cd \teta$ vanishes on tensors of the
  form $\xi_u\odot \xi_v$ with $x\neq y$. Clearly we can choose a
  basis of $S^2 H^1(C,T_C) $ formed by such elements.
\end{proof}

\begin{say}
  Since the second fundamental form is symmetric, using this theorem
  we get another proof of Lemma \ref{sym}.
\end{say}

\section{Totally geodesic submanifolds and gonality}

\begin{say}
  In this section we will give an upper bound for the dimension of a
  germ of a totally geodesic submanifold contained in the Jacobian
  locus.
\end{say}
\begin{teo}
  \label{rank}
  Assume that $C$ is a $k$-gonal curve of genus $g$, with $g\geq 4$
  and $k \geq 3$. Then there exists a quadric $Q \in I_2(K_C)$ such
  that $\operatorname{rank}\rho(Q) > 2g-2k$.
\end{teo}
\proof Here we follow the simplified notation mentioned in
\ref{rem-notazione}. So we understand that a local coordinate has been
fixed at the relevant points and we write $\xi_P$ for a Schiffer
variation at $P$.  Let $F$ be a line bundle on $C$ such that $|F|$ is
a $g^1_k$ and choose a basis $\{x,y\}$ of $H^0(F)$.  Set $M = K_C -F$
and denote by $B$ the base locus of $|M|$. By Clifford theorem
$\deg(B) < k-2$.  Assume that $\langle t_1,t_{2}\rangle$ is a pencil
in $H^0(M)$, with base locus $B$.  We can write $t_i = t'_i s$ for a
section $s\in H^0(C,\OO_C(B)) $ with $\operatorname{div} (s) = B$.
Then $\sx t'_1, t'_2\xs $ is a base point free pencil in $|M-B|$.  Let
$\psi : C \ra \PP^1$ be the morphism induced by this pencil.

Now consider the rank 4 quadric $Q:= xt_1 \odot yt_2 - xt_2 \odot
yt_1$. Clearly $Q \in I_2(K_C)$. Set $d:= \deg(M-B) = 2g-2-k-
\deg(B)$.  We need the following fact: if
  $\{P_1,...,P_d\}$ is a fibre of the morphism $\psi$ over a regular
value, then the Schiffer variations $\xi_{P_1},...,
\xi_{P_d}$ are linearly independent in $H^1(C,T_C)$.  In fact, denote
by $W:= \langle \xi_{P_1},..., \xi_{P_d}\rangle$. We want to show that
the annihilator $\Ann(W)$ of $W$ has dimension $3g-3-d$. By lemma
\ref{cauchy}, $\Ann(W)= \{ \alpha \in H^0(C,2K_C) \ | \ \alpha(P_i)
=0,\ i=1,...,d \}$, hence by Riemann-Roch $\dim\Ann(W) =
h^0(2K_C-P_1-\cdots -P_d) = h^0(2K_C-M+B) = g-1+k + \deg(B) = 3g-3
-d$.  The claim is proven.  Next denote by $\phi$ the morphism induced
by the pencil $|F|$ and consider the set $E:= \psi(\Crit(\phi) \cup
\Crit(\psi) \cup B)$ where $\Crit(\phi) $ (resp.  $\Crit(\psi) $)
denote the set of critical points of $\phi$ (resp. $\psi$).  Let $z
\in {\PP}^1 \setminus E$ and let
$\{P_1,\ldots,P_d\}$ be the fibre of
$\psi$ over $z$. By changing coordinates on $\PP^1$ we can assume $z =
[0,1]$, i.e. $t'_1(P_i) = 0$ for $i =1,\ldots d$. Then clearly
$t_1(P_i) = 0$, so $Q(P_i,P_j) = 0$ for all $i,j$.  Applying Theorem
\ref{copito}, one immediately obtains that the restriction of
$\rho(Q)$ to the subspace $W$ is represented in the basis
$\{\xi_{P_1},..., \xi_{P_d}\}$ by a diagonal matrix with entries $\pi
i \cd \mu_2(Q)(P_i)$ on the diagonal.  For a rank 4 quadric the second
Gaussian map can be computed as follows: $ \mu_2(Q) = \mu_{1,F}(x
\wedge y)\mu_{1,M}(t_1 \wedge t_2)$ see \cite[Lemma 2.2]{cf2}.  Now
$\mu_{1,F}(x\wedge y) (P_i) \neq 0$, because $P_i \not \in
\Crit(\phi)$ by the choice of $z$, see \eqref{muformula}.  Moreover
$P_i \not \in B$.  On $C\setminus B$ the morphism $\psi$ coincides
with the map associated to $\sx t_1, t_2\xs$. Since $P_i \not \in
\Crit(\psi)$, it is not a critical point for the latter map. Therefore
also $\mu_{1,M}(t_1\wedge t_2) (P_i) \neq 0$ see \eqref{muformula}.
Thus $\mu_2(Q)(P_i) = \mu_{1,F}(x \wedge y)(P_i)\mu_{1,M}(t_1 \wedge
t_2)(P_i) \neq 0$ for every $i=1,\lds, d$.  This shows that in the
basis $\{\xi_{P_1},..., \xi_{P_d}\}$ the quadric $\rho(Q)\restr{W}$ is
represented by a diagonal matrix with non-zero diagonal entries.  So
$\rho(Q)$ has rank at least $d = 2g-2-k-\deg B > 2g-2k$.\qed

\begin{teo}
  \label{stima1}
  Assume that $C$ is a $k$-gonal curve of genus $g$ with $g\geq 4$ and
  $k\geq 3$. Let $Y$ be a germ of a totally geodesic submanifold of
  $A_g$ which is contained in the jacobian locus and passes through
  $[C]$. Then $\dim Y \leq 2g+k - 4$.
\end{teo}
\proof By Theorem \ref{rank} we know that there exists a quadric $Q
\in I_2$ such that the rank of $\rho(Q)$ is at least $2g-2k+1$. By
assumption for any $v \in T_{[C]}Y$ we must have that $\rho(Q)(v \odot
v) = 0$, so $v$ is isotropic for $\rho(Q)$, hence
\begin{gather*}
  \dim T_{[C]} Y \leq 3g-3 - \frac{(2g-2k+1) }{2}= 2g+ k - \frac{7}{2}
  .
\end{gather*}
\qed
\begin{remark}
  In Theorem \ref{rank} if $|M|$ is base point free -- this happens in
  particular if $C$ is generic in the locus of $k$-gonal curves -- the
  above proof shows that $ \rank \rho(Q) \geq 2g-2-k$.  So if $Y$ is a
  germ of a totally geodesic submanifold of $A_g$ contained in the
  jacobian locus and passing through a generic $k$-gonal curve, then
  $\dim Y \leq 2g-2 + {k}/{2}$.
\end{remark}

\begin{teo}
  \label{stima2}
  If $g\geq 4$ and $Y$ is a germ of a totally geodesic submanifold of
  $A_g$ contained in the jacobian locus, then $\dim Y \leq
  \frac{5}{2}(g-1)$.
\end{teo}
\proof This immediately follows from Theorem \ref{stima1}, since the
gonality of a genus $g$ curve is at most $ [(g+3)/{2}]$.  \qed

\begin{cor}
  For any $g\geq 4$ and $k \geq 2$ the $k$-gonal locus is not totally
  geodesic in $A_g$.
\end{cor}
\begin{proof}
  For $k =2$ this is Proposition 5.1 in \cite{cf1}.  If $k \geq 3$ the
  dimension of the $k$-gonal locus is $2g + 2k-5 > 2g +k -4$. Hence
  the statement follows immediately from Theorem \ref{stima1}.
\end{proof}

\begin{remark}
  As it is evident from the proof, gonality is used to construct a
  quadric $Q \in I_2(K_C)$ of rank 4 such that $\rho(Q)$ has large
  rank. It seems unlikely that gonality plays any role in this
  problem.  In fact we expect the existence of $Q \in I_2(K_C)$ with
  image $\rho(Q)$ a nondegenerate quadric on $H^1(C,T_C)$. This would
  give the upper bound $\frac{3}{2}(g-1)$ for the dimension of a germ
  of a totally geodesic submanifold.  On the other hand, the
    map $\rho$ is injective. This can be deduced from
    \cite[Cor. 3.4]{cf1} or from Theorem \ref{prodotto}, since the
    form $\teta$ is non-zero. Therefore $\rho(I_2(K_C))$ is a linear
    system of quadrics of dimension $\frac{(g-1)(g-4)}{2}$ on
    $\PP(H^1(C,T_C)) =\PP^{3g-4}$.  This already gives an upper bound
    for the dimension of a submanifold $Y$ as in Theorem
    \ref{stima2}. Indeed for any point $[C]\in Y$, the tangent space
    $T_{[C]} Y \subset H^1(C,T_C)$ is contained in the base locus of
    $\rho(I_2(K_C))$.  This means that $\rho(I_2(K_C))$ is contained
    in the space of quadrics $ q \in S^2H^0(C,2K_C)$ that vanish on
    $T_{[C]}Y$. This yields the bound
    \begin{gather*}
      \dim Y \leq \frac{ -1 + \sqrt{ 32g^2 -40g +1}} {2}.
    \end{gather*}
    Nevertheless for any $g\geq 2$ this bound is weaker than the one
    provided in Theorem \ref{stima2}.  At any case the study of the
    base locus of the linear system $\rho(I_2(K_C))$ should clearly
    improve the understanding of totally geodesic submanifolds.  If
    one could prove that the base locus is empty, one would rule out
    the existence of totally geodesic submanifolds passing through
    $[C]$. However to proceed in this direction it is probably
    necessary to better understand the form $\teta$.
\begin{remark}
  Observe that for $g \geq 5$, the non existence of germs of totally
  geodesic hypersurfaces follows directly from Theorem \ref{copito}
  and \cite[Thm.  6.1]{cf2}.  In fact if $Y$ is a hypersurface in
  $M_g$, it passes through a non-trigonal curve $[C]$. Since $\PP
  T_{[C]} Y$ intersects the bicanonical curve in $\PP H^1(T_C)$,
  $T_{[C]} Y$ contains a Schiffer variation $\xi_x$, for some $x \in
  C$. By \cite[Thm. 6.1]{cf2}, there exists a quadric $Q \in I_2(K_C)$
  such that $\mu_2(Q)(x) \neq 0$, so $\rho(Q)(\xi_x \odot \xi_x) \neq
  0$ by Theorem \ref{copito}.
\end{remark}

\end{remark}

\section{Families of cyclic covers of the projective line}

\begin{say}
  Let $C$ be a curve of genus $g \geq 4$ and let $G$ be
  a subgroup of the group of automorphisms of $C$. This yields an
  inclusion of $G$ in the mapping class group $\Gamma_g$
  \cite{gavino}.  So $G$ acts on the Teichm\"uller space $T_g$ and we
  denote by $T_g^G$ the set of fixed points of $G$ on $T_g$, which is
  a nonempty by the solution of Nielsen realization problem
  \cite{ke,tro,ca}, and is a complex submanifold of $T_g$.  The
  tangent space to $T_g^G$ at a point $[C]$ is given by $H^1(C,
  T_{C})^{G}$.  Moreover $T_g^G$ parametrizes marked curves $C$
  endowed with a holomorphic action of $G$ of the given topological
  type and there is a universal family ${\mathcal C} \ra T_g^G$
  containing all such curves. If $t\in T_g^G$ we denote by $C_t$ the
  corresponding curve.
\end{say}

\begin{say}\label{rhotilde}
  Let us now consider the period map at the level of Teichm\"uller
  spaces $T_g \ra H_g$. This is an immersion outside the hyperelliptic
  locus, and we will consider its restriction to $T_g^G$. Denote by
  $Z(G)$ its image in $A_g$.  Given a point $t \in T_g^G$ such that
  $C_t$ is non-hyperelliptic, the cotangent spaces fit in the exact
  sequence
  \begin{equation}\nonumber
    0 \ra N^* \ra S^2 H^0(K_{C_t})  \stackrel{\pi}{\ra} H^0 (2K_{C_t})^{G} \ra 0,
  \end{equation}
  where $N^*$ is the conormal space to $Z(G) \subset A_g$ at the point
  $J(C_t)$.  Since $\pi$ is $G$-invariant, $N^*$ is a $G$-submodule.
  Let
  \begin{equation}\nonumber
    \tilde{\rho}: N^* \ra H^0 (2K_{C_t})^{G}  \otimes H^0 (2K_{C_t})^{G} 
  \end{equation}
  denote the second fundamental form of $Z(G)$ at point $J(C_t)$.
  
\end{say}

\begin{lemma}
  \label{tro-equiv}
  The second fundamental form $\tilde{\rho}$ is $G$-equivariant.
 
\end{lemma}
\begin{proof}
  Recall that by definition, given $a \in N^*$, we have
  $\tilde{\rho}(a) = \pi ( \nabla(a))$, where $\nabla$ is the
  Gauss-Manin connection on $S^2f_*\om_{\mathcal{C}/T_g^G}$ of the
  family $\mathcal C\stackrel{f}{ \ra} T_g^G$.  Since $\pi$ is
  $G$-invariant, it suffices to prove that $\nabla$ is $G$-invariant,
  i.e. $g\meno \nabla g = \nabla$ for every $g\in G$. Observe that the
  flat connection $\nabla^{GM}$ on $R^1 f_* \C$ is $G$-invariant,
  since the $G$--action maps flat sections to flat sections.  Let $
  \pi^{1,0}: H^1(C_t, \C) \ra H^{1,0}(C_t)$ be the projection. Then
  $\pi^{1,0} \circ \nabla^{GM}$ is the connection on the Hodge bundle
  $ f_*\om_{\mathcal{C}/T_g^G}$.  Since the action of $G$ on
  ${\mathcal C}$ is holomorphic, the projection $ \pi^{1,0}$ is
  $G$-equivariant, hence $\pi^{1,0} \circ \nabla^{GM}$ is also a
  $G$-invariant connection.  Finally $\nabla $ is the connection
  induced by $\pi^{1,0} \circ \nabla^{GM}$ on
  $S^2f_*\om_{\mathcal{C}/T_g^G}$.
    The result follows. 
\end{proof}

\begin{prop}
  \label{suffmoo}
  If there are no nonzero quadrics in $I_2(K_{C_t})$, which are
  invariant under the action of the group $G$, then $Z(G)$ is totally
  geodesic.
\end{prop}
\proof Consider the cotangent sequence of the period map $j: { T}_g
\rightarrow { H}_g$ at the point $[C_t]$ and its restriction to
$T_g^G$:
\begin{gather}
  \label{diatro}
  \begindc{\commdiag}[4]
  \obj(0 ,5) [a1]{$0$} \obj(12,5)[a2]{$ I_2(K_{C_t})$}
  \obj(30,5)[a3]{$S^2H^0( K_{C_t})$} \obj(50,5)[a4]{$H^0(2K_{C_t})$}
  \obj(65,5)[a5]{$0$} \obj(0, -5) [b1]{$0$} \obj(12,-5)[b2]{$ N^*$}
  \obj(30,-5)[b3]{$S^2H^0( K_{C_t})$}
  \obj(50,-5)[b4]{$H^0(2K_{C_t})^G$} \obj(65,-5)[b5]{$0.$}
  \mor{a1}{a2}{} \mor{b1}{b2}{} \mor{a4}{a5}{} \mor{b4}{b5}{}
  \mor{a2}{a3}{} \mor{b2}{b3}{} \mor{b3}{b4}{} \mor{a3}{a4}{$m$}
  \mor{a3}{b3}{$=$} \mor{a2}{b2}{}[-1,5] \mor{a4}{b4}{}[-1,7]
  \enddc
\end{gather}
%
(The notation is as in \ref{rhotilde}.)  Since the map $m$ is
$G$-equivariant, any $G$-invariant element in $N^*$ lies in
$I_2(K_{C_t})$.  Hence by the assumption there are no nontrivial
invariant elements in $N^*$, i.e. the trivial representation does not
appear in the decomposition of $N^*$.  On the other hand the
representation $H^0(2K_{C_t})^G$ is trivial.  By Lemma \ref{tro-equiv}
$\tro$ is a morphism of $G$-representations.  Hence Schur lemma
implies that $\tro$ is the trivial map.  \qed

\begin{say}
  Now we will consider the case in which $G = \Zeta/m\Zeta$, $m \geq
  3$ and $C/G \cong \PP^1$.  These families have been studied by
  various authors, e.g. \cite{dejong-noot,moo,rohde}, since they
  provide the only known examples of totally geodesic submanifolds
  contained in the Jacobian locus.  These examples are in fact Shimura
  varieties. A complete list of the Shimura varieties that can be
  obtained in this way has been given in \cite{moo}, which we follow
  for the notation in the rest of the paper.

  We identify $G$ also with the group of $m$-th roots of unity.  Fix
  an integer $N \geq 4$, together with an $N$-tuple $a=(a_1,...,a_N)$
  of positive integers, such that $\gcd(m,a_1,...,a_N) =1$, $a_i \not
  \equiv 0 \mod m$ and $\sum_{i=1}^N a_i \equiv 0 \mod m$. Given
  distinct points $t_1,...,t_N \in \PP^1$ there is a well-defined
  curve $C_t$ which is a cyclic cover of $\PP^1$ with covering group
  $\Zeta/m\Zeta$, branch points $t_i$ and local monodromy $a_i$ at
  $t_i$.  It is the normalization of the affine curve
  \begin{equation}
    \label{cyclic}
    y^m = \prod_{i=1}^N (x-t_i)^{a_i}.
  \end{equation} 
  Varying the branch points $t=(t_1, \lds, t_N)$ one obtains a
  $(N-3)$-dimensional family of curves $\mathcal{C} \ra B$.  There is
  an action of $G$ on $\mathcal{C}$ given by the rule $ \zeta \cd
  (x,y,t) := (x, \zeta \cdot y,t)$, $ \zeta \in G$.  Thus triples
  $(m,N,a)$ parametrize these families and two triples $(m,N,a)$ and
  $(m',N',a')$ yield equivalent families if and only if $m=m'$, $N =
  N'$ and the classes of $a$ and $a'$ in $(\Zeta/m \Zeta)^N$ are in
  the same orbit under the action of $(\Zeta/m \Zeta)^* \times
  \Sigma_N$, where $(\Zeta/m \Zeta)^*$ acts diagonally by
  multiplication and the symmetric group $\Sigma_N$ acts by
  permutation of the indices.  Notice that to fix the class of the
  triple $(m,N,a)$ is equivalent to fixing the topological type of the
  $G$-action.  Thus $B$ and $T_g^G$ have the same image in $M_g$ and
  we will consider $Z(m,N,a):=Z(G) \subset A_g$ the
  $(N-3)$-dimensional subvariety given by the Jacobians of the curves
  $C_t$, $t\in B$.  By the Hurwitz formula
  \begin{gather*}
    g = 1 + \frac{(N-2)m - \sum_{i=1}^N r_i}{2}
  \end{gather*}
  with $r_i:=\gcd(m,a_i)$.  A basis of $H^0(C_t, K_{C_t})$ is given as
  follows. For $i \in \{1,...,N\}$ and $n \in \Zeta$ set
  \begin{gather*}
    l(i,n):= \Bigl [{\frac{-na_i}{m}} \Bigr]\qquad \quad d_n = -1 +
    \sum_{i=1}^N \Bigl \sx \frac{-na_i}{m}\Bigr \xs
  \end{gather*}
  (Here $[ a ]$ and $\sx a \xs $ denote the integral and fractional
  parts of $a \in \R$.)  Since $G$ acts on $C_t$, there is a
  decomposition $H^0(C_t, K_{C_t}) = \oplus_{n=0}^{m-1} V_n$, where
  $V_n$ is the subspace of 1-forms $\om$ such that $\zeta \cd \om =
  \zeta^n \om$.  Then $V_0=\{0\}$, while for $n =1, \lds, m-1$ a basis
  of $V_n$ is provided by the forms that have the following expression
  in the model \eqref{cyclic}:
  \begin{equation}
    \label{forms}
    \omega_{n,\nu} := y^n \cdot (x-t_1)^{\nu} \cdot
    \prod_{i=1}^N (x-t_i)^{l(i,n)} \cdot dx \qquad  0 \leq \nu \leq d_n -1.
  \end{equation}
\end{say}

\begin{remark}
  In \cite{moo}, Moonen proved that there is a finite list of triples
  $(m,N,a)$ such that the corresponding subvariety $Z = Z(m, N,a)
  \subset A_g$ is a Shimura variety.  By \cite{moonen-JAG} a Shimura
  variety is a totally geodesic algebraic submanifold of $A_g$ that
  contains a complex multiplication point.  Therefore the families in
  Moonen's list are totally geodesic.  Using Proposition \ref{suffmoo}
  we can immediately verify that all these families are totally
  geodesic. In fact in all those cases there are no nontrivial
  quadrics in $I_2(K_{C_t}) $ which are invariant under the action of
  the cyclic group.
\end{remark}

\begin{prop}
  \label{suffmoo1}
  Consider the family ${\mathcal C} \rightarrow B$ associated to the
  triple $(m,N,a)$ as above and the corresponding subvariety $Z = Z(m,
  N,a) \subset A_g$ given by the Jacobians of the curves $C_t$.
  Assume that not all the curves $C_t$ are hyperelliptic.  If there
  exists an integer $n \in \{1,...,m-1\}$ such that $d_n \geq 2$,
  $d_{m-n} \geq 2$ and $n \neq m-n$, then $Z$ is not totally geodesic.
\end{prop}
\proof Fix an arbitrary non-hyperelliptic curve $C=C_t$ belonging to
the family.  We use the representation \eqref{cyclic}.  Since $d_n
\geq 2$ and $d_{m-n} \geq 2$ and $n \neq m-n$, there are four distinct
forms $\omega_{n,0}, \ \omega_{n,1}, \ \omega_{m-n,0}, \
\omega_{m-n,1}$ given in \eqref{forms}. We form the quadric
\begin{gather*}
  Q := \omega_{n,0} \odot \omega_{m-n,1} - \omega_{n,1} \odot
  \omega_{m-n,0}.
\end{gather*}
One immediately sees from the definition that $Q \in I_2(K_{C})$ and
that $Q$ is $G$-invariant.  Let $D $ be the divisor of poles of the
meromorphic function $x \in \mathcal{M}(C)$.  Let $\sigma_0, \sigma_1
\in H^0(C, \OO_{C}(D))$ be the sections corresponding to the
meromorphic functions $1$ and $x$.  Assume for simplicity that $t_1
=0$, so we have $\omega_{n,1} = x \cdot \omega_{n,0}$ and
$\omega_{m-n,1} = x \cdot \omega_{m-n,0}$.  Hence the forms
$\omega_{n,0}, \omega_{m-n,0}$ can be seen as elements in $H^0(C,
K_{C}( -D))$.  Set $\tau_0 := \omega_{n,0}$, $ \tau_1:=
\omega_{m-n,0}$. The quadric $Q$ can be written as follows:
\begin{equation}
  \label{Qpencil}
  Q= \sigma_0 \tau_0 \odot \sigma_1 \tau_1- \sigma_0 \tau_1 \odot \sigma_1 \tau_0. 
\end{equation}
Denote by $\phi: C \rightarrow \PP^1$ our covering: it corresponds to
the pencil $\langle 1,x\rangle = \langle \sigma_0, \sigma_1
\rangle$. Denote by $\psi$ the map to $\PP^1$ induced by the other
pencil $\langle \tau_0, \tau_1 \rangle$. Take a point $p \in C$ that
does not belong to the set $\Crit (\phi) \cup G\cd \Crit(\psi) \cup
B$, where $B$ is the base locus of $\psi$.  Fix a nonzero vector $u
\in T_pC$.  The vector $v: = \sum_{g \in G} \xi_{dg(u)} \in H^1
(C,T_{C})$ is clearly $G$-invariant.  So it is a tangent vector to
$T_g^G$ at the point corresponding to $C$.  Hence by diagram
\eqref{diatro} we have $\rho(Q)(v \odot v) = \tro(Q)(v \odot v)$. By
Lemma \eqref{invarho} the map $\rho$ is also $G$-equivariant. So,
using Theorem \ref{copito} we get
\begin{gather*}
  \tro(Q)(v \odot v) = \rho(Q)(v \odot v) = \sum_{g_1,g_2 \in G}
  \rho(Q)(\xi_{dg_1(u)} \odot \xi_{dg_2(u)})= \\
  =|G| \cdot \bigg{(}\sum_{g \neq 1} \rho(Q)(\xi_u \odot \xi_{dg(u)})
  -2 \pi
  i \mu_2(Q)(p)\bigg{)}\\
  = -2\pi i m \cdot \bigg{(}2 \sum_{g \neq 1} Q(u, dg(u)) \cdot
  \eta_u(dg(u))+ \mu_2(Q)(p)\bigg{)}.
\end{gather*}
By \eqref{Qpencil} $Q$ is the quadric associated to the pencils $\sx
\sigma_0, \sigma_1\xs$ and $\sx \tau_0, \tau_1\xs$, corresponding to
the maps $\phi$ and $\psi$ respectively. Since the fibre of $\phi$
containing $p$ is the orbit $G\cd p$, $Q( u, g_*u) = 0$, forall $g \in
G \setminus\{1\}$.  So finally we get
\begin{gather*}
  \tro(Q)(v \odot v) = -2 \pi i m \cdot \mu_2(Q)(p),
\end{gather*}
and $\mu_2(Q)(p) \neq 0$ by our choice of $p$.  Hence $Z$ is not
totally geodesic.  \qed

\begin{cor}
  \label{suffmoo2}
  For any $m$, there is only a finite number of families which can be
  totally geodesic.
\end{cor}
\proof By Proposition \ref{suffmoo1}, if there exists an integer $n
\in \{1,..., m-1\}$ such that $d_n, d_{m-n} \geq 2$, $n \neq n-m$,
then the family is not totally geodesic. So we can assume that for all
$n\in \{1,..., m-1\}$ either $d_n \leq 1$, or $d_{m-n} \leq 1$. In
particular either $d_1 \leq 1$, or $d_{m-1} \leq 1$.  Denote by $N_i$
the cardinality of the set $\{j \ | \ a_j =i\}$. Then
$\sum_{i=1}^{m-1} N_i = N$ and the relation $\sum_{i=1}^N a_i \equiv
0$ mod m becomes $\sum_{i=1}^{m-1} i N_i \equiv 0$ mod m.  But
\begin{gather*}
  d_l = -1 + \sum_{i=1}^N\Big \langle \frac{-l a_i}{m}\Big \rangle= -1
  + \sum_{i =1}^{m-1} N_i\Big \langle \frac{-l i}{m}
  \Big \rangle, \quad \text{for }l=1,\lds, m-1,\\
  \text{so} \quad d_1 = -1 + \sum_{i=1}^{m-1} \frac{m-i}{m} N_i, \quad
  d_{m-1} = -1 + \sum_{i=1}^{m-1} \frac{i}{m} N_i.
\end{gather*}
Hence for $l \in \{1, m-1\}$ we have $d_l \geq -1 + \sum_{i =1}^{m-1}
\frac{N_i}{m} = -1 + \frac{N}{m}$. So $d_l \leq 1$ implies $N \leq
2m$. Hence only a finite number of families can be totally geodesic.
\qed
  
Now we give some examples of computations for low degree ($m=3$ and
$m=5$), showing that by the previous results most of the families are
not totally geodesic.

\begin{cor}
  If $m =3$ and $g \geq 5$, then the varieties $Z(3,N,a)$ are not
  totally geodesic. If $g =4$ there is only one totally geodesic
  family given by the triple $(3,6,a)$ where $a = (1,1,1,1,1,1)$.
\end{cor}
\proof If $d_1,d_2 \geq 2$, we know by Proposition \ref{suffmoo1} that
the family is not totally geodesic. So we can assume that there exists
$n \in \{1,2\}$ such that $d_n = 1$. In this case the space $S^2
H^0(K_{C_t})^{(0)}$ given by the invariant elements equals $\langle
\omega_{n,0} \odot V_{3-n}\rangle$.  So there are no nonzero invariant
elements in $I_2(K_{C_t})$ and by Proposition \ref{suffmoo} we
conclude that the family is totally geodesic.  So we have to show that
$g=4$, $N =6$ and $a = (1,1,1,1,1,1)$. Denote as above by $N_i$ the
cardinality of the set $\{j \ | \ a_j =i\}$. Then $N_1 + N_2 = N$ and
$N_1 + 2N_2 \equiv 0$ mod 3.  So
\begin{gather*}
  d_1 = -1 + \sum_{i=1}^N\Big \langle \frac{-a_i}{3}\Big \rangle = -1
  + N_1 \Big\langle \frac{-1}{3}\Big \rangle+ N_2 \Big\langle
  \frac{-2}{3}\Big \rangle
  = -1+  \frac{2}{3} N_1+ \frac{1}{3} N_2,\\
  d_2 = -1 + \sum_{i=1}^N\Big \langle \frac{-2a_i}{3}\Big \rangle = -1
  + N_1 \Big\langle \frac{-2}{3}\Big \rangle+ N_2\Big \langle
  \frac{-4}{3}\Big \rangle = -1+ \frac{1}{3} N_1+ \frac{2}{3} N_2.
\end{gather*}
By Hurwitz formula $g = N-2$. Since we are assume $g\geq 4$ we get $
N_1 + N_2 = N \geq 6$.  If $d_1 = 1$,
$$6 \leq  N_1 + N_2 \leq 2N_1 + N_2 = 6,$$
so $N = 6$, $g =4$, $N_1 =0$, $N_2 =6$, hence $a=(2,2,2,2,2,2)$ which
is equivalent to $(1,1,1,1,1,1)$.  If $d_2 = 1$ we have
$$6 \leq N_1 + N_2 \leq N_1 + 2N_2 = 6,$$
so $N=6$, $g=4$, $N_2=0$, $N_1=6$, $a= (1,1,1,1,1,1)$.  \qed

\begin{remark}
  If $m=5$, the following families are totally geodesic (see the list
  in \cite{moo}):
  \begin{enumerate}
  \item $g=4$, $N=4$, $a=(1,3,3,3)$,
  \item $g= 6$, $N=5$, $a=(2,2,2,2,2)$.
  \end{enumerate}
  (1) is the family constructed by de Jong-Noot \cite{dejong-noot}.
  Applying the criterion given in Proposition \ref{suffmoo1} we are
  able to prove that all other families are not totally geodesic
  except possibly for the following 4 cases
  \begin{enumerate}
  \item [(3)] $g=4$, $N=4$, $a=(1,1,4,4)$,
  \item [(4)] $g=4$, $N=4$, $a=(1,2,3,4)$,
  \item [(5)]$g=6$, $N=5$, $a=(1,1,1,3,4)$,
  \item [(6)]$g=6$, $N=5$, $a=(1,1,2,2,4)$,
  \item [(7)] $g=8$, $N=6$, $a=(1,1,2,2,2,2)$.
  \end{enumerate}
  (3) is contained in the hyperelliptic locus, see \cite[(5.7)]{moo}.
  (4) has been studied in detail in \cite{vangeemen}. Since one can
  check that it contains a CM point, it is not totally geodesic by
  Moonen's classification.
  
  It would be interesting to get a complete list of the families of
  cyclic coverings which are totally geodesic and to compare it with
  the list in \cite{moo}.
\end{remark}
%
%
%
%
%
%
%
%
%
%
%
%

\end{document}